\documentclass[a4paper]{article}
\usepackage{fullpage}
\usepackage{amsmath}
\usepackage{amssymb}
\usepackage{amsthm}
\usepackage{graphicx}
\usepackage{bbm}
\usepackage{enumerate}
\usepackage{microtype}
\usepackage{xcolor}
\usepackage[none]{hyphenat}
\usepackage[colorlinks=true,urlcolor=blue,linkcolor=blue,plainpages=false,pdfpagelabels]{hyperref}
\allowdisplaybreaks

\newtheorem{theorem}{Theorem}[section]
\newtheorem{proposition}[theorem]{Proposition}
\newtheorem{lemma}[theorem]{Lemma}
\newtheorem{corollary}[theorem]{Corollary}

\newtheorem{definition}[theorem]{Definition}

\newcommand{\N}{\mathbb{N}}

\newcommand{\R}{\mathbb{R}}
\newcommand{\eps}{\varepsilon}

\newcommand{\wt}[1]{\widetilde{#1}}

\DeclareMathOperator{\sgn}{sgn}

\title{Rates of metastability for iterations on the unit interval}
\author{Andrei Sipo\c s${}^{a,b}$\\[2mm]
\footnotesize ${}^a$Research Center for Logic, Optimization and Security (LOS), Department of Computer Science,\\
\footnotesize Faculty of Mathematics and Computer Science, University of Bucharest,\\
\footnotesize Academiei 14, 010014 Bucharest, Romania\\[1mm]
\footnotesize ${}^b$Simion Stoilow Institute of Mathematics of the Romanian Academy,\\
\footnotesize Calea Grivi\c tei 21, 010702 Bucharest, Romania\\[2mm]
\footnotesize E-mail: andrei.sipos@fmi.unibuc.ro\\
}
\date{}

\begin{document}

\maketitle

\begin{abstract}
We use techniques of proof mining to extract computable and uniform rates of metastability (in the sense of Tao) for iterations of continuous functions on the unit interval, firstly (following earlier work of Gaspar) out of convergence proofs due to Franks, Marzec, Rhoades and Hillam and then out of an argument due to Borwein and Borwein that pertains only to Lipschitz functions.


\end{abstract}

\section{Introduction}

In 1953, Mann \cite{Man53} introduced a new iterative process based on lower triangular matrices with the goal of adapting classical techniques like Ces\`aro summation to processes of successive approximation e.g. in Banach spaces. Later (see \cite{OutGro69,Dot70}), it became clear that a special role is played by the particular case of ``segmenting'' or ``normal'' iterations, i.e. sequences $(x_n)$ where, given a self-mapping $f$ and a parameter sequence $(t_n) \subseteq [0,1]$, for all $n$,
$$x_{n+1}=(1-t_n)x_n + t_nf(x_n).$$
Where the parameter sequence is constant, such iterations were also independently studied by Krasnoselski in 1955 \cite{Kra55} (and later also by Schaefer \cite{Sch57} and Browder and Petryshyn \cite{BroPet66}) and they are thus called (in the general, non-constant case) {\it Krasnoselski-Mann iterations}.

In the 1970s, people started to be interested in obtaining results on such iterations on the unit interval $[0,1]$. In 1974, Bailey \cite{Bai74} showed that Krasnoselski's theorem on finding a fixed point of a nonexpansive self-mapping of a compact set of a uniformly convex Banach space can be proven in an elementary way when the set is the unit interval, and then Hillam \cite{Hil75} generalized this result to Lipschitz self-mappings. As we said, in this case, the parameter sequence is constant and thus at each step the same operation is applied to the sequence, i.e. one essentially deals with Picard iterates. This led to Hillam's 1976 result \cite{Hil76} which showed that the convergence of Picard iterates on the unit interval for arbitrary continuous functions is equivalent to asymptotic regularity (a concept introduced in \cite{BroPet66}), i.e. to the fact that $d(x_n,x_{n+1}) \to 0$.

General continuous functions on the unit interval had already been considered starting with Franks and Marzec in 1971 \cite{FraMar71}, who proved a convergence theorem for the Krasnoselski-Mann iterates associated to such a function in the case where for all $n$, $t_n=1/(n+1)$, a result which Rhoades \cite{Rho74a} would generalize to more general $(t_n)$ and then in \cite{Rho76} to the case of {\it Ishikawa iterations}. (Shiro Ishikawa had introduced in 1974 \cite{Ish74} these iterations in order to overcome the shortcomings of the Krasnoselski-Mann process when dealing with the class of Lipschitz pseudocontractions; later, Chidume and Mutangadura \cite{ChiMut01} illustrated their necessity by exhibiting an example of a Lipschitzian pseudocontractive self-mapping of the two-dimensional disc with a unique fixed point for which no Krasnoselski-Mann sequence converges.)

Our work in this area serves primarily as a case study in {\it proof mining}, an applied subfield of mathematical logic that aims to analyse proofs in ordinary mathematics in order to obtain additional information which is not immediately apparent. Proof mining in its current form has been developed in the last twenty years primarily by Ulrich Kohlenbach and his collaborators (see \cite{Koh08} for a comprehensive monograph; a recent survey which serves as a short and accessible introduction is \cite{Koh19}), and one of its main achievements consists in the so-called `general logical metatheorems' that guarantee under certain conditions the extractability of information of a quantitative nature which we shall now detail.

Let $l \in \R$ and $(x_n)$ be a sequence of real numbers. We say that $(x_n)$ {\it tends towards $l$ with rate of convergence $\beta : (0, \infty) \to \N$} if for all $\delta >0$ and all $n \geq \beta(\delta)$, we have $|x_n -l| \leq \delta$. This is clearly a quantitative variant of convergence where $\beta(\delta)$ is an upper bound for the point of convergence $N$ in the usual formulation. One could also produce a similar definition for a rate corresponding to the Cauchy property of a sequence, and then easily compute one rate from the other.

Ideally, one would want to obtain uniform and computable rates of convergence for iterations, but even for simple cases like monotone sequences in the unit interval where one has a uniform {\it proof} of convergence, one cannot possibly have a uniform {\it rate} since it is almost immediate that convergence may be {\it arbitrarily slow}. We also mention the related phenomenon of {\it Specker sequences} \cite{Spe49}: convergent and computable sequences of reals having no computable rate of convergence (see \cite[Theorem 4.4.2]{Koh19w} for an example adequate to the case at hand; one may however have computable rates of convergence under more restrictive assumptions like the existence of a modulus of uniqueness, see \cite{Lam05}).

The next best thing is then a finitary notion of convergence, introduced by Terence Tao in \cite{Tao08} (and used successfully in his proof of the convergence of multiple ergodic averages \cite{Tao08A}), usually called {\it metastability} (under a suggestion of Jennifer Chayes), which is formulated as follows for a given sequence of reals $(x_n)$:
$$\forall \eps>0 \,\forall g:\N\to\N \,\exists N\in \N \,\forall i,j\in [N,N+g(N)] \ \left(| x_i-x_j| \leq\eps\right),$$
a property which is easily (but non-constructively) seen to be equivalent to $(x_n)$ being Cauchy. Because of its reduced logical complexity, the metatheorems mentioned above make it possible to extract a computable and uniform {\it rate of metastability} -- a bound $\Theta(\eps,g)$ on the $N$ in the sentence above -- from any proof that shows the convergence of a given class of sequences and that may be formalized in one of the logical systems for which such metatheorems have so far been developed (one may however need to add certain constants to the system which will then manifest as additional parameters of the rate).

Tao's resulting finite monotone convergence principle may be found in Section~\ref{sec:prelim}, along with other preliminary notions. In addition to that, a significant number of rates of metastability have already been extracted out of a variety of convergence proofs in nonlinear analysis, ergodic theory and convex optimization (see again \cite{Koh19}). In the case of iterations on the unit interval, the most notable work so far has been the one of Jaime Gaspar in his 2011 PhD thesis \cite{Gas11}. There, he analyzed Hillam's equivalence result in \cite{Hil76} mentioned above in order to derive a rate of metastability for the sequence $(x_n)$ conditional on a rate of convergence towards $0$ for the sequence $d(x_n,x_{n+1})$. His main achievement was to fit the original proof of Hillam into a system of lower logical strength by replacing the use of the Bolzano-Weierstrass theorem with that of the infinite pigeonhole principle, thus resulting in a rate of low computational complexity.

In Section~\ref{sec:cont}, we build on Gaspar's work in order to obtain rates of metastability for the Mann (Theorem~\ref{mann}) and Ishikawa (Theorem~\ref{ishikawa}) iterations of arbitrary continuous functions in the unit interval. In the special case of $t_n=1/(n+1)$ due to Franks and Marzec, one has a rate of metastability which is unconditional, in the sense that it depends in addition to $\eps$ and $g$ only on a modulus of uniform continuity for the self-mapping (to be defined in the next section), and not on a modulus corresponding to some post hoc property of the iterative sequence.

In the above, we use throughout the formulation of Park \cite{Par81}, who provided a unifying framework that encompasses all the above-mentioned results. There is one outlier, though, namely the generalization of Hillam's earlier results in \cite{Hil75} on Lipschitz self-mappings due to Borwein and Borwein \cite{BorBor91}, whose proof contains a kind of argument that has never been analyzed so far using the tools of proof mining. This one we treat in Section~\ref{sec:lip}.

\section{Preliminaries}\label{sec:prelim}

After introducing some notations, we shall be in position to present Tao's finitary analysis of the monotone convergence principle as an illustrating example of obtaining a rate of metastability which will also be useful later.

For all $g: \N \to \N$, we define $\wt{g} : \N \to \N$, for all $n$, by $\wt{g}(n):=n+g(n)$. Also, for all $f:\N \to \N$ and all $n \in \N$, we denote by $f^{(n)}$ the $n$-fold composition of $f$ with itself. Note that for all $g$ and $n$, $\wt{g}^{(n)}(0)\leq\wt{g}^{(n+1)}(0)$.

\begin{proposition}[Finite Monotone Convergence Principle {\cite{Tao08}}]\label{fmcp}
Let $\eps>0$, $g:\N \to \N$. Let $(a_i)_{i=0}^{\wt{g}^{\left(\left\lceil\frac1\eps\right\rceil + 1\right)}(0)}$ be a finite monotone sequence in $[0,1]$. Then there is an $N \leq \wt{g}^{\left(\left\lceil\frac1\eps\right\rceil \right)}(0)$ with $N+g(N) \leq\wt{g}^{\left(\left\lceil\frac1\eps\right\rceil + 1\right)}(0)$  such that for all $i$, $j \in [N,N+g(N)]$, $|a_i-a_j|\leq\eps$.
\end{proposition}

\begin{proof}
Assume w.l.o.g. that $(a_i)$ is nonincreasing. Assume that the conclusion is false, hence in particular for all $i \leq \left\lceil\frac1\eps\right\rceil$, $a_{\wt{g}^{(i)}(0)} - a_{\wt{g}^{(i+1)}(0)} > \eps$. Then
$$a_0 \geq a_0 - a_{\wt{g}^{\left(\left\lceil\frac1\eps\right\rceil + 1\right)}(0)} = \sum_{i=0}^{\left\lceil\frac1\eps\right\rceil} \left(a_{\wt{g}^{(i)}(0)} - a_{\wt{g}^{(i+1)}(0)} \right)> \left\lceil\frac1\eps\right\rceil\cdot\eps \geq 1,$$
a contradiction.
\end{proof}

This immediately gives us a uniform and computable rate of metastability for monotone sequences in the unit interval.

\begin{corollary}\label{mmcp}
Let $(a_n)$ be a monotone sequence in $[0,1]$. Then for all $\eps>0$ and $g:\N\to\N$ there is an $N \leq \wt{g}^{\left(\left\lceil\frac1\eps\right\rceil \right)}(0)$ such that for all $i$, $j \in [N,N+g(N)]$, $|a_i-a_j|\leq\eps$.
\end{corollary}

Even though, as we said in the Introduction, we cannot hope to obtain rates of convergence for the iterations which we shall discuss, we shall use such rates (as they apply to sequences of parameters used in the construction of the iterations) as part of the data in terms of which our rates of metastability will be defined. Another notion that we shall need in this vein will be that of a {\it modulus of uniform continuity} for a function $f: [0,1] \to [0,1]$, which is a function $\omega : (0, \infty) \to (0,\infty)$ such that for any $\delta >0$ and any $x$, $y \in [0,1]$ with $|x-y| < \omega(\delta)$, we have that $|f(x)-f(y)| < \delta$. Clearly, a function $f: [0,1] \to [0,1]$ has a modulus of uniform continuity if and only if it is uniformly continuous.

We shall also need the following elementary inequality concerning real numbers.

\begin{lemma}\label{l2}
For all $w$, $x$, $y$, $z \in \R$, $|x-y| \geq |w-z| - |w-x| - |y-z|$.
\end{lemma}

\begin{proof}
By the triangle inequality, one has $|w-z| \leq |w-x|+|x-y|+|y-z|$.
\end{proof}

\section{Continuous functions}\label{sec:cont}

The following theorem provides rates of metastability corresponding to convergence theorems that concern iterations of general continuous functions (on the unit interval) that fall into the Krasnoselski-Mann scheme. We generally follow the ideas of Gaspar \cite{Gas11}; the main differences are that the sequence needs now to pass three times (labelled below $j_0$, $j_1$ and $j_2$) through the two intervals considered in the proof (labelled by $i_0$ and $i_1$) and that one introduces an extra use of $g$ between these times (seen below in the formula for $u$).

\begin{theorem}\label{mann}
Let $f: [0,1] \to [0,1]$. Let $(x_n)$ be a sequence in $[0,1]$ such that for each $n$, $x_{n+1}$ is between $x_n$ and $f(x_n)$.

Define, for any suitable $\eps$, $g$, $\omega$, $\beta$, $p$, $n$:
\begin{align*}
m_\eps&:=\left\lceil\frac6\eps\right\rceil\\
c_\eps&:=\frac1{4m_\eps}\\
A_{\eps,g}(p)&:=\frac1{\max\left(1,12m_\eps g(p)\right)}\\
C_{\eps,g,\omega}(p)&:=\min\left(A_{\eps,g}(p),\omega\left(A_{\eps,g}(p)\right)\right)\\
u^{\eps,g,\omega,\beta}_0&:=\beta\left(c_\eps\right)\\
u^{\eps,g,\omega,\beta}_{n+1}&:=\max\left(u^{\eps,g,\omega,\beta}_n+g\left(u^{\eps,g,\omega,\beta}_n\right)+1,\beta\left(C_{\eps,g,\omega}\left(u^{\eps,g,\omega,\beta}_n\right)\right)\right)\\
\Phi^{\rm KM}_{\omega,\beta}(\eps,g)&:=u^{\eps,g,\omega,\beta}_{2m_\eps^2}.
\end{align*}
Let $\omega: (0, \infty) \to (0, \infty)$ be such that $\omega$ is a modulus of uniform continuity for $f$ and $\beta : (0, \infty) \to \N$ be such that $(x_n - x_{n+1})$ tends towards $0$ with rate of convergence $\beta$.

Let $\eps > 0$ and $g: \N \to \N$. Then there is an $N \leq \Phi^{\rm KM}_{\omega,\beta}(\eps,g)$ such that for all $i$, $j \in [N, N+ g(N)]$, $|x_i - x_j| \leq \eps$.
\end{theorem}

\begin{proof}
We may now drop $\eps$, $g$, $\omega$, $\beta$ where they show up as indices or arguments. It is immediate that:
\begin{itemize}
\item for all $n$, $u_n \leq u_n + g(u_n) < u_{n+1}$;
\item for all $n$, $\beta(c) \leq u_n$;
\item for all $n$, $\beta(C(u_n))\leq u_{n+1}$.
\end{itemize}
Assume by way of contradiction that for all $N \leq \Phi^{\rm KM}$  there are $i$, $j \in [N,N +g(N)]$ with $|x_i-x_j| > \eps$, so for all $N \leq \Phi^{\rm KM}$ there is an $i \in (N,N+g(N)]$ with $|x_N-x_i|>\eps/2$ -- in particular $g(N) > 0$. Put, for each such $N$, $H(N)$ to be the least $i$ with this property.

Denote, for each $n \in \{0,\ldots,m-1\}$, $I_n:=[\frac n m,\frac{n+1}m]$. Then, by the pigeonhole principle, there are $j_0 < j_1 < j_2 \leq 2m^2$ and $i_0$, $i_1 \leq m$ -- assume w.l.o.g. that $i_0 \leq i_1$ -- such that $x_{u_{j_0}}$, $x_{u_{j_1}}$, $x_{u_{j_2}} \in I_{i_0}$ and $x_{H\left(u_{j_0}\right)}$, $x_{H\left(u_{j_1}\right)}$, $x_{H\left(u_{j_2}\right)} \in I_{i_1}$. Since, by the definition of $H$, $|x_{u_{j_0}}-x_{H\left(u_{j_0}\right)}| >\eps/2$, we have that $i_1 - i_0 \geq 2$. Thus, $I_{i_0+1}$ is an interval between and distinct from $I_{i_0}$ and $I_{i_1}$.

We now distinguish two cases.\\[2mm]

{\bf Case I.} For all $x$ in the middle half of $I_{i_0+1}$, $|fx-x| < \frac1{4mg\left(u_{j_0}\right)}$.\\[1mm]

Since $\beta(c) \leq u_{j_0}$, we have that for all $n \in [u_{j_0}, H(u_{j_0}))$, $|x_n - x_{n+1}| \leq c = 1/(4m)$. Thus, considering that $x_{u_{j_0}} \in I_{i_0}$ and $x_{H(u_{j_0})} \in I_{i_1}$, there is a $k\in [u_{j_0}, H(u_{j_0}))$ such that $x_k$ is in the second quarter of $I_{i_0+1}$, and take $k$ to be biggest with this property, so that there is no $l \in [k,H(u_{j_0})]$ with $x_l$ to the left of the middle half of $I_{i_0+1}$. Then, for all $l \in [k,H(u_{j_0}))$ such that $x_l$ is in the middle half of $I_{i_0+1}$, $|x_{l+1} - x_l| \leq |f(x_l) - x_l| \leq \frac1{4mg\left(u_{j_0}\right)}$. Since $H(u_{j_0}) - k \leq (u_{j_0} + g(u_{j_0})) - u_{j_0} = g(u_{j_0})$, we have by induction that for all $l \in (k,H(u_{j_0})]$, $x_l$ is in the third quarter of  $I_{i_0+1}$, contradicting the fact that $x_{H(u_{j_0})} \in I_{i_1}$.\\[2mm]

{\bf Case II.} There is an $x$ in the middle half of $I_{i_0+1}$ with $|fx-x| \geq \frac1{4mg\left(u_{j_0}\right)}$.\\[1mm]

Put $J:=[x-C(u_{j_0}),x+C(u_{j_0})]$. Then, by the definitions of $C$ and $A$, for all $y$ in the interior of $J$, $|y-x| < \frac1{12mg\left(u_{j_0}\right)}$ and $|fx-fy| < \frac1{12mg\left(u_{j_0}\right)}$, so, by using Lemma~\ref{l2},
$$|fy-y| \geq |fx-x| - |fx-fy| - |y-x| \geq \frac1{12mg\left(u_{j_0}\right)},$$
so for all $y \in J$, by the continuity of $f$,
$$|fy-y| \geq  \frac1{12mg\left(u_{j_0}\right)}.$$

Clearly, $J$ is entirely contained within the interior of $I_{i_0+1}$. We now distinguish two sub-cases.

{\bf Sub-case 1.} We have $fx > x$. (Note that we cannot assume w.l.o.g. that $fx>x$, since we shall use the fact that $fx-x$ has the same sign as $i_1-i_0$.)

Then for all $y \in J$, $fy>y$. We show that for each $n \geq H(u_{j_1})$, $x_n \geq x$. If $n=H(u_{j_1})$, this is immediate since then $x_n \in I_{i_1}$. Assume now $x_n \geq x$ for an $n \geq H(u_{j_1})$. Since $n \geq H(u_{j_1}) \geq u_{j_1} \geq \beta(C(u_{j_0}))$, $|x_n-x_{n+1}| \leq C(u_{j_0})$. If $x_n \geq x + C(u_{j_0})$, then, from $x_n \leq |x_n - x_{n+1}| + x_{n+1}$, we get that
$$x_{n+1} \geq x_n - |x_n - x_{n+1}| \geq x + C(u_{j_0}) - C(u_{j_0}) = x.$$
If $x_n<x + C(u_{j_0})$, then $x_n \in J$, so in that case, since $f(x_n) > x_n$, $x_{n+1} \geq x_n$, so $x_{n+1} \geq x$.

Since $H(u_{j_1}) \leq u_{j_1} + g(u_{j_1}) < u_{j_2} $, we have that $x_{u_{j_2}} \geq x$, contradicting the fact that $ x_{u_{j_2}} \in I_{i_0}$.

{\bf Sub-case 2.} We have $fx < x$. This sub-case follows roughly in the same way as sub-case 1, with $u_{j_1}$ replacing $H(u_{j_1})$ and $H(u_{j_1})$ replacing $u_{j_2}$.
\end{proof}

In the case where for all $n$, $x_{n+1}=f(x_n)$, the above gives a rate of metastability for the theorem of Hillam in \cite{Hil76}, a rate which is slightly more complicated than the one previously extracted by Gaspar \cite{Gas11}. In the case of the Krasnoselski-Mann iteration, i.e. where, given a parameter sequence $(t_n) \subseteq [0,1]$, for all $n$, $x_{n+1}=(1-t_n)x_n + t_nf(x_n)$, Rhoades \cite{Rho74a} showed convergence under the assumption $t_n \to 0$. Since here $x_n - x_{n+1} = t_n(x_n - f(x_n))$, a rate of convergence for $(t_n)$ towards $0$ is also a rate of convergence for $(x_n - x_{n+1})$ towards $0$, and thus our result covers this case. A particular case of that was first treated by Franks and Marzec \cite{FraMar71}, namely the case when for all $n$, $t_n = 1/(n+1)$. There, a rate of convergence for $(t_n)$ may be taken to be $\delta \mapsto \left\lceil\frac1\delta\right\rceil$, since for all $\delta>0$ and all $n \geq \left\lceil\frac1\delta\right\rceil$,
$$\frac1{n+1} \leq \frac1{\left\lceil\frac1\delta\right\rceil +1} \leq\delta.$$
Thus, in this case one obtains an {\bf unconditional} rate of metastability for the iterative sequence.

Rhoades has also considered \cite{Rho76} the case of the Ishikawa iteration -- i.e. where, given two parameter sequences $(t_n)$, $(s_n) \subseteq [0,1]$, for all $n$, $x_{n+1}=(1-t_n)x_n + t_nf(s_nf(x_n) + (1-s_n)x_n)$ -- for which we shall now extract a rate of metastability. (Since this configuration generalizes the one in Theorem~\ref{mann}, we could have given Theorem~\ref{mann} as a corollary of Theorem~\ref{ishikawa} below, but the particularization of the rate of metastability would not have removed the extraneous complications introduced by this more general case.)

\begin{theorem}\label{ishikawa}
Let $f: [0,1] \to [0,1]$. Let $(x_n)$ and $(y_n)$ be sequences in $[0,1]$ such that for each $n$, $y_n$ is between $x_n$ and $f(x_n)$ and $x_{n+1}$ is between $x_n$ and $f(y_n)$.

Define, for any suitable $\eps$, $g$, $\omega$, $\beta$, $\gamma$, $p$, $n$:
\begin{align*}
m_\eps&:=\left\lceil\frac6\eps\right\rceil\\
c_\eps&:=\frac1{4m_\eps}\\
B_{\eps,g}(p)&:=\frac1{\max\left(1,8m_\eps g(p)\right)}\\
Z_{\eps,g,\omega}(p)&:=\min\left(B_{\eps,g}(p),\omega\left(B_{\eps,g}(p)\right)\right)\\
C_{\eps,g,\omega}(p)&:=\min\left(\frac{Z_{\eps,g}(p)}3,\omega\left(\frac{Z_{\eps,g}(p)}3\right)\right)\\
u^{\eps,g,\omega,\beta,\gamma}_0&:=\beta\left(c_\eps\right)\\
u^{\eps,g,\omega,\beta,\gamma}_{n+1}&:=\max\left(u^{\eps,g,\omega,\beta,\gamma}_n+g\left(u^{\eps,g,\omega,\beta,\gamma}_n\right)+1,\beta\left(\frac{C_{\eps,g,\omega}\left(u^{\eps,g,\omega,\beta,\gamma}_n\right)}2\right),\gamma\left(\frac{C_{\eps,g,\omega}\left(u^{\eps,g,\omega,\beta,\gamma}_n\right)}2\right)\right)\\
\Phi^{\rm I}_{\omega,\beta,\gamma}(\eps,g)&:=u^{\eps,g,\omega,\beta,\gamma}_{2m_\eps^2}.
\end{align*}
Let $\omega: (0, \infty) \to (0, \infty)$ be such that $\omega$ is a modulus of uniform continuity for $f$, $\beta : (0, \infty) \to \N$ be such that $(x_n - x_{n+1})$ tends towards $0$ with rate of convergence $\beta$ and $\gamma : (0, \infty) \to \N$ be such that $(x_n - y_n)$ tends towards $0$ with rate of convergence $\gamma$.

Let $\eps > 0$ and $g: \N \to \N$. Then there is an $N \leq \Phi^{\rm I}_{\omega,\beta,\gamma}(\eps,g)$ such that for all $i$, $j \in [N, N+ g(N)]$, $|x_i - x_j| \leq \eps$.
\end{theorem}

\begin{proof}
We may now drop $\eps$, $g$, $\omega$, $\beta$, $\gamma$ where they show up as indices or arguments. It is immediate that:
\begin{itemize}
\item for all $n$, $u_n \leq u_n + g(u_n) < u_{n+1}$;
\item for all $n$, $\beta(c) \leq u_n$;
\item for all $n$, $\beta\left(\frac{C(u_n)}2\right)\leq u_{n+1}$ and $\gamma\left(\frac{C(u_n)}2\right)\leq u_{n+1}$.
\end{itemize}
Assume by way of contradiction that for all $N \leq \Phi^{\rm I}$  there are $i$, $j \in [N,N +g(N)]$ with $|x_i-x_j| > \eps$, so for all $N \leq \Phi^{\rm I}$ there is an $i \in (N,N+g(N)]$ with $|x_N-x_i|>\eps/2$ -- in particular $g(N) > 0$. Put, for each such $N$, $H(N)$ to be the least $i$ with this property.

Denote, for each $n \in \{0,\ldots,m-1\}$, $I_n:=[\frac n m,\frac{n+1}m]$. Then, by the pigeonhole principle, there are $j_0 < j_1 < j_2 \leq 2m^2$ and $i_0$, $i_1 \leq m$ -- assume w.l.o.g. that $i_0 \leq i_1$ -- such that $x_{u_{j_0}}$, $x_{u_{j_1}}$, $x_{u_{j_2}} \in I_{i_0}$ and $x_{H\left(u_{j_0}\right)}$, $x_{H\left(u_{j_1}\right)}$, $x_{H\left(u_{j_2}\right)} \in I_{i_1}$. Since, by the definition of $H$, $|x_{u_{j_0}}-x_{H\left(u_{j_0}\right)}| >\eps/2$, we have that $i_1 - i_0 \geq 2$. Thus, $I_{i_0+1}$ is an interval between and distinct from $I_{i_0}$ and $I_{i_1}$. Note also that $B\left(u_{j_0}\right) = \frac1{8mg\left(u_{j_0}\right)}$.

We now distinguish two cases.\\[2mm]

{\bf Case I.} For all $x$ in the middle half of $I_{i_0+1}$, $|fx-x| < Z\left(u_{j_0}\right)$.\\[1mm]

Since $\beta(c) \leq u_{j_0}$, we have that for all $n \in [u_{j_0}, H(u_{j_0}))$, $|x_n - x_{n+1}| \leq c = 1/(4m)$. Thus, considering that $x_{u_{j_0}} \in I_{i_0}$ and $x_{H(u_{j_0})} \in I_{i_1}$, there is a $k\in [u_{j_0}, H(u_{j_0}))$ such that $x_k$ is in the second quarter of $I_{i_0+1}$, and take $k$ to be biggest with this property, so that there is no $l \in [k,H(u_{j_0})]$ with $x_l$ to the left of the middle half of $I_{i_0+1}$. Let $l \in [k,H(u_{j_0}))$ be such that $x_l$ is in the middle half of $I_{i_0+1}$. Then $|f(x_l) - x_l| < B\left(u_{j_0}\right)$ and  $|f(x_l) - x_l| < \omega\left( B\left(u_{j_0}\right)\right)$, so $|y_l-x_l| <\omega\left( B\left(u_{j_0}\right)\right)$ and
$$|x_{l+1} - x_l| \leq |f(y_l) - x_l| \leq |f(y_l) -f(x_l)| + |f(x_l) - x_l| < B\left(u_{j_0}\right) + B\left(u_{j_0}\right) = \frac1{4mg\left(u_{j_0}\right)}.$$
Since $H(u_{j_0}) - k \leq (u_{j_0} + g(u_{j_0})) - u_{j_0} = g(u_{j_0})$, we have by induction that for all $l \in (k,H(u_{j_0})]$, $x_l$ is in the third quarter of  $I_{i_0+1}$, contradicting the fact that $x_{H(u_{j_0})} \in I_{i_1}$.\\[2mm]

{\bf Case II.} There is an $x$ in the middle half of $I_{i_0+1}$ with $|fx-x| \geq Z\left(u_{j_0}\right)$.\\[1mm]

Put $J:=[x-C(u_{j_0}),x+C(u_{j_0})]$. Then, by the definitions of $C$, $Z$ and $B$, for all $y$ in the interior of $J$, $|y-x| < \frac{Z\left(u_{j_0}\right)}3$ and $|fx-fy| < \frac{Z\left(u_{j_0}\right)}3$, so, by using Lemma~\ref{l2},
$$|fy-y| \geq |fx-x| - |fx-fy| - |y-x| \geq \frac{Z\left(u_{j_0}\right)}3,$$
so for all $y \in J$, by the continuity of $f$,
$$|fy-y| \geq  \frac{Z\left(u_{j_0}\right)}3.$$

Clearly, $J$ is entirely contained within the interior of $I_{i_0+1}$. We now distinguish two sub-cases.

{\bf Sub-case 1.} We have $fx > x$.

Then for all $y \in J$, $fy>y$. We show that for each $n \geq H(u_{j_1})$, $x_n \geq x$. If $n=H(u_{j_1})$, this is immediate since then $x_n \in I_{i_1}$. Assume now $x_n \geq x$ for an $n \geq H(u_{j_1})$. Since $n \geq H(u_{j_1}) \geq u_{j_1} \geq \beta\left(\frac{C(u_{j_0})}2\right)$, $|x_n-x_{n+1}| \leq \frac{C(u_{j_0})}2$ and similarly $n \geq \gamma\left(\frac{C(u_{j_0})}2\right)$, so $|x_n-y_n| \leq \frac{C(u_{j_0})}2$.

If $x_n \geq x + \frac{C(u_{j_0})}2$, then, from $x_n \leq |x_n - x_{n+1}| + x_{n+1}$, we get that
$$x_{n+1} \geq x_n - |x_n - x_{n+1}| \geq x + \frac{C(u_{j_0})}2 - \frac{C(u_{j_0})}2 = x.$$
If $x_n<x + \frac{C(u_{j_0})}2$, then $x_n \in J$, so in that case $fx_n > x_n$, so $fx_n \geq y_n \geq x_n$. Since $|x_n-y_n| \leq \frac{C(u_{j_0})}2$, $y_n \leq x_n + \frac{C(u_{j_0})}2 \leq x+C(u_{j_0})$, so $y_n \in J$. We get $fy_n \geq y_n$, so $fy_n \geq x_n$ and $x_{n+1} \geq x_n \geq x$.

Since $H(u_{j_1}) \leq u_{j_1} + g(u_{j_1}) < u_{j_2}$, we have that $x_{u_{j_2}} \geq x$, contradicting the fact that $ x_{u_{j_2}} \in I_{i_0}$.

{\bf Sub-case 2.} We have $fx < x$. This sub-case follows roughly in the same way as sub-case 1, with $u_{j_1}$ replacing $H(u_{j_1})$ and $H(u_{j_1})$ replacing $u_{j_2}$.
\end{proof}

If the self-mapping of the unit interval is monotone, Rhoades showed -- for the Mann iteration in \cite{Rho74b} and for the Ishikawa iteration in \cite{Rho76} -- that without imposing any additional condition on the sequences of parameters, the iterative sequence is also monotone. Thus, in these situations, Corollary~\ref{mmcp} provides a rate of metastability.

\section{Lipschitz functions}\label{sec:lip}

In \cite{Hil75}, Hillam showed that the Krasnoselski-Mann iteration for a Lipschitz function of Lipschitz constant $L>0$ with constant parameter $1/(L+1)$ is monotone, hence convergent. He also stated without proof that convergence also holds for a constant parameter strictly smaller than $2/(L+1)$. Later, Borwein and Borwein showed in \cite{BorBor91} that monotonicity holds if the parameters are all smaller than $1/(L+1)$ and convergence holds if the parameters are bounded away from $2/(L+1)$. In the sequel, we analyze this latter proof to obtain a rate of metastability. (The monotone cases are covered by Corollary~\ref{mmcp}.)

The following lemmas, together with their proofs, are adapted from \cite{BorBor91}.

\begin{lemma}\label{dl1}
Let $L >0$, $f: [0,1] \to [0,1]$ be $L$-Lipschitz, $x$, $x^* \in [0,1]$, $\delta \in (0,1)$ and $t \in [0,1]$ such that $t \leq \frac{2-\delta}{L+1}$ and $x^*=(1-t)x + tf(x)$. Let $p$ be a fixed point of $f$ which is located between $x$ and $x^*$. Then
$$|x^*-p| \leq (1-\delta)|x-p|.$$
\end{lemma}

\begin{proof}
Assume w.l.o.g. $x \leq x^*$. Then
\begin{align*}
|x^*-p|&=x^*-p = (1-t)(x-p) + t(f(x)-f(p)) \leq (t-1)(p-x) + tL(p-x) \\
&= (t(1+L)-1)(p-x) \leq (1-\delta)|x-p|.
\end{align*}
\end{proof}

\begin{definition}
Let $(x_n) \subseteq [0,1]$ and $f: [0,1] \to [0,1]$.

We say that $(\sigma_n) \subseteq \{\pm1\}$ is the {\bf sign sequence} for $(x_n)$ relative to $f$ if $\sigma_0=1$ and for all $n$, if $f(x_n)-x_n\neq 0$, $\sigma_{n+1}=\sgn(f(x_n)-x_n)$ and otherwise $\sigma_{n+1}=\sigma_n$ -- note that for all $n$, if $\sigma_{n+1}=1$ (respectively $-1$), then $f(x_n)-x_n \geq 0$ (respectively $\leq 0$).

We say that $(q_n) \subseteq \N \cup \{\infty\}$ is the {\bf switching sequence} for $(x_n)$ relative to $f$ if, denoting by $(\sigma_n)$ the sign sequence for $(x_n)$ relative to $f$, $q_0=0$ and for all $n$, if $q_n=\infty$ then $q_{n+1}=\infty$ else if there is a $k> q_n$ with $\sigma_{k+1} = -\sigma_{q_n+1}$, $q_{n+1}$ is the least such $k$, else $q_{n+1}=\infty$ -- note that for all $n$ with $q_{n+1} < \infty$, we have that $\sigma_{q_{n+1}+1}=-\sigma_{q_n+1}$ and that for all $l \in [q_n+1,q_{n+1}]$, $\sigma_l=\sigma_{q_n+1}$.
\end{definition}

\begin{lemma}\label{dl3}
Let $L >0$, $f: [0,1] \to [0,1]$ be $L$-Lipschitz, $(t_n)$ and $(x_n)$ be sequences in $[0,1]$ such that for all $n$, $x_{n+1}=(1-t_n)x_n + t_nf(x_n)$. Let $(q_n)$ be the switching sequence for $(x_n)$ relative to $f$. Let $r \geq 1$ with $q_{r+1} < \infty$ and put $n_1:=q_r-1$ and $n_2:=q_{r+1}-1$. Let $\delta \in (0,1)$ be such that for all $n$, $t_n \leq \frac{2-\delta}{L+1}$. Then:
\begin{enumerate}[(i)]
\item for all $n \in [n_1+1,n_2+1]$, $x_n$ is located between $x_{n_1}$ and $x_{n_1+1}$;\label{i1}
\item $|x_{n_2}-x_{n_2+1}| \leq \left(1- \frac\delta2\right)|x_{n_1}-x_{n_1+1}|$.\label{i2}
\end{enumerate}
\end{lemma}

\begin{proof}
Let $(\sigma_n)$ be the sign sequence for $(x_n)$ relative to $f$. Assume w.l.o.g. that $\sigma_{n_1+2}=-1$, so $\sigma_{n_1+1}=\sigma_{n_2+2}=1$ and for all $n \in [n_1+2,n_2+1]$, $\sigma_n=-1$. Then $f(x_{n_1}) - x_{n_1} \geq 0 \geq f(x_{n_1+1}) - x_{n_1+1}$ and $x_{n_1} \leq x_{n_1+1}$, so there is a fixed point of $f$ in $[x_{n_1},x_{n_1+1}]$. Let $p$ be the least one (using here the continuity of $f$). By Lemma~\ref{dl1}, $x_{n_1+1}-p \leq (1-\delta)(p-x_{n_1}) \leq p-x_{n_1}$, which may also be written as $x_{n_1+1}-p \leq \frac12(x_{n_1+1}-x_{n_1})$. Note that $(x_n)$ is nonincreasing between $n_1+1$ and $n_2+1$.\\[1mm]

\noindent {\bf Claim 1.} We have that $x_{n_2} \geq p$.\\[1mm]
\noindent {\bf Proof of claim 1:} Assume that $p>x_{n_2}$. Then there is an $n' \in [n_1+1,n_2)$ with $x_{n'} \geq p > x_{n'+1}$. By Lemma~\ref{dl1}, we have that
$$p- x_{n'+1} \leq (1-\delta)(x_{n'}-p) \leq (1-\delta)(x_{n_1+1}-p) \leq (1-\delta)(p-x_{n_1}),$$
so
$$x_{n'+1} \geq (1-\delta)x_{n_1} + \delta p \geq x_{n_1}.$$
Since $n'+2 \leq n_2+1$, $\sigma_{n'+2}=-1$, so $f(x_{n'+1})-x_{n'+1} \leq 0$. Then, since $f(x_{n_1}) - x_{n_1} \geq 0$, there is a fixed point $q$ between $x_{n_1}$ and $x_{n'+1}$. But then $x_{n_1} \leq q \leq x_{n'+1} < p \leq x_{n_1+1}$, which contradicts the minimality of $p$.
\hfill $\blacksquare$\\[2mm]

Thus, using that either $x_{n_2+1} \geq p$ or $p \geq x_{n_2+1}$, we have that either $x_{n_2+1} \geq p \geq x_{n_1}$ or $x_{n_2} \geq p \geq x_{n_2+1}$.\\[1mm]

\noindent {\bf Claim 2.} We have that $x_{n_2 +1} \geq (1-\delta)x_{n_1} + \delta p$.\\[1mm]
\noindent {\bf Proof of claim 2:} In the first case above, the statement is obvious. Suppose now that $x_{n_2} \geq p \geq x_{n_2+1}$. Then, by Lemma~\ref{dl1},
$$p-x_{n_2+1} \leq (1-\delta)(x_{n_2}-p),$$
so
$$x_{n_2+1} \geq (2-\delta)p - (1-\delta)x_{n_2}.$$
It remains to be shown that $(2-\delta)p - (1-\delta)x_{n_2} \geq (1-\delta)x_{n_1} + \delta p$. Since $p-x_{n_1} \geq x_{n_1+1} - p \geq x_{n_2} - p$, $2p-x_{n_2} \geq x_{n_1}$, which we multiply by $(1-\delta)$ to obtain the desired inequality.
\hfill $\blacksquare$\\[2mm]

Let $n \in [n_1+1,n_2+1]$. Then $x_{n_1+1} \geq x_n \geq x_{n_2+1} \geq (1-\delta)x_{n_1} + \delta p \geq x_{n_1}$. Thus, we get (\ref{i1}).

We now prove (\ref{i2}). We have that
\begin{align*}
x_{n_2} - x_{n_2+1} &\leq x_{n_1+1} - x_{n_2+1} \leq (1-\delta)(x_{n_1+1}-x_{n_1}) + \delta(x_{n_1+1}-p) \\
&\leq  (1-\delta)(x_{n_1+1}-x_{n_1}) + \frac\delta2(x_{n_1+1}-x_{n_1}) = \left(1- \frac\delta2\right)(x_{n_1+1}-x_{n_1}).
\end{align*}
\end{proof}

We may now state and prove the corresponding metastability theorem.

\begin{theorem}
Let $L >0$, $f: [0,1] \to [0,1]$ be $L$-Lipschitz, $(t_n)$ and $(x_n)$ be sequences in $[0,1]$ such that for all $n$, $x_{n+1}=(1-t_n)x_n + t_nf(x_n)$.

Define, for any suitable $\eps$, $g$, $\delta$, $m$, $n$:
\begin{align*}
h^g_m(n)&:=g(m+n)\\
P^{\eps,g}_0&:=0\\
P^{\eps,g}_{n+1}&:=P^{\eps,g}_n + \wt{h^g_{P^{\eps,g}_n}}^{\left(\left\lceil\frac1\eps\right\rceil+1\right)}(0)\\
T_{\eps,\delta}&:= \left\lceil \log_{\left(1-\frac\delta2\right)}\eps \right\rceil +1\\
B_{\eps,g,\delta}&:= T_{\eps,\delta} + \wt{g}\left(P^{\eps,g}_{T_{\eps,\delta}}\right)  + 1\\
\Psi^{\rm KM}_\delta(\eps,g)&:=P^{\eps,g}_{B_{\eps,g,\delta}}.
\end{align*}

Let $\delta \in (0,1)$ be such that for all $n$, $t_n \leq \frac{2-\delta}{L+1}$.

Let $\eps > 0$ and $g: \N \to \N$. Then there is an $N \leq \Psi^{\rm KM}_\delta(\eps,g)$ such that for all $i$, $j \in [N, N+ g(N)]$, $|x_i - x_j| \leq \eps$.
\end{theorem}

\begin{proof}
We may now drop $\eps$, $g$, $\delta$ where they show up as indices or arguments. It is immediate that:
\begin{itemize}
\item for all $n$, $P_n \leq P_{n+1}$;
\item $\left(1-\frac\delta2\right)^{T-1} \leq\eps$.
\end{itemize}
Let $(q_n)$ be the switching sequence for $(x_n)$ relative to $f$. Note that $(q_n)$ is strictly increasing and for all $r$, $r \leq q_r$ and if $q_r < \infty$, then $(x_n)_{n\in[q_r,q_{r+1})}$ is monotone. We distinguish two cases.\\[2mm]

{\bf Case I.} There is an $r \leq B$ with $q_r > P_r =P_{r-1} + \wt{h_{P_{r-1}}}^{\left(\left\lceil\frac1\eps\right\rceil+1\right)}(0)$.\\[1mm]

Take $r$ to be minimal with this property. Clearly, $r \geq 1$ and $q_{r-1} \leq P_{r-1}$, so
$$(x_{P_{r-1}+i})_{i=0}^{ \wt{h_{P_{r-1}}}^{\left(\left\lceil\frac1\eps\right\rceil+1\right)}(0)}$$
is a subsequence of $(x_n)_{n\in[q_{r-1},q_r)}$ and is thus monotone. By Proposition~\ref{fmcp}, there is an $N' \leq \wt{h_{P_{r-1}}}^{\left(\left\lceil\frac1\eps\right\rceil\right)}(0)$ such that for all $i$, $j \in [P_{r-1} + N', P_{r-1} + N'+h_{P_{r-1}}(N')]$, $|x_i - x_j| \leq \eps$.

Put $N:=P_{r-1}+N'$. Then
$$N \leq P_{r-1}+ \wt{h_{P_{r-1}}}^{\left(\left\lceil\frac1\eps\right\rceil\right)}(0) \leq P_{r-1}+ \wt{h_{P_{r-1}}}^{\left(\left\lceil\frac1\eps\right\rceil + 1\right)}(0) = P_r \leq P_B = \Psi^{\rm KM}.$$
In addition,
$$h_{P_{r-1}}(N') = g(P_{r-1}+N') = g(N),$$
so for all $i$, $j \in [N, N+ g(N)]$, $|x_i - x_j| \leq \eps$.\\[2mm]

{\bf Case II.} For all $r \leq B$, $q_r \leq P_r$.\\[1mm]

We first show that for all $r \in [1,B-1]$ and all $n \in [q_r,q_B]$, $x_n$ is between $x_{q_r-1}$ and $x_{q_r}$. Let $r \in [1,B-1]$. We prove that for all $s \in [r,B-1]$ and all $n \in [q_s,q_{s+1}]$, $x_n$ is between $x_{q_r-1}$ and $x_{q_r}$. If $s=r$, this follows immediately from Lemma~\ref{dl3}.(\ref{i1}). Now let $s \geq r+1$. By the induction hypothesis, for all $m \in [q_{s-1},q_s]$, $x_m$ is between $x_{q_r-1}$ and $x_{q_r}$ -- in particular, $x_{q_s-1}$ and $x_{q_s}$ are. By Lemma~\ref{dl3}.(\ref{i1}), $x_n$ is between $x_{q_s-1}$ and $x_{q_s}$, thus also between $x_{q_r-1}$ and $x_{q_r}$.

By Lemma~\ref{dl3}.(\ref{i2}), we get that for all $r \in [1,B-1]$, $|x_{q_{r+1}-1} - x_{q_{r+1}}|\leq \left(1-\frac\delta2\right)|x_{q_r-1} - x_{q_r}|$ and thus, by an easy induction, for all $r \in [1,B-1]$, $|x_{q_r-1} - x_{q_r}| \leq \left(1-\frac\delta2\right)^{r-1}$. Combining this with the result in the previous paragraph, we get that for all $r \in [1,B-1]$ and all $i$, $j \in [q_r,q_B]$, $|x_i-x_j| \leq \left(1-\frac\delta2\right)^{r-1}$.

Since $T \leq B-1$, for all $i$, $j \in [q_T,q_B]$, $|x_i-x_j| \leq \eps$. Take $N:=P_T \leq P_B = \Psi^{\rm KM}$. Then on one hand $N=P_T \geq q_T$ (since $T \leq B$) and on the other $N+ g(N) = \wt{g}(P_T) \leq T + \wt{g}(P_T) + 1 = B \leq q_B$, so $[N,N+g(N)] \subseteq [q_T,q_B]$, hence for all $i$, $j \in [N, N+ g(N)]$, $|x_i - x_j| \leq \eps$.
\end{proof}

An examination of the above proof shows that the only properties that are needed about the iterative sequence are contained in Lemma~\ref{dl3}. In \cite{DD92}, it is shown that one can impose certain conditions on the parameters of the Ishikawa iteration such that those same properties hold and thus the rate extracted above remains valid.

\section{Acknowledgements}

This work has been supported by the German Science Foundation (DFG Project KO 1737/6-1) and by a grant of the Romanian National Authority for Scientific Research, CNCS - UEFISCDI, project number PN-III-P1-1.1-PD-2019-0396.

\end{document}